\documentclass[oneside,11pt]{amsart}
\usepackage{amssymb,amscd,amsmath,latexsym}
\usepackage[mathcal]{euscript}

\newtheorem{thm}{Theorem}[subsection]
\newtheorem{cor}[thm]{Corollary}
\newtheorem{lem}[thm]{Lemma}
\newtheorem{prop}[thm]{Proposition}
\theoremstyle{definition}

\newtheorem{exm}[thm]{Example}
\newtheorem{conj}[thm]{Conjecture}
\newtheorem{prob}[thm]{Problem}

\theoremstyle{remark}
\newtheorem{rem}[thm]{Remark}

\newcommand{\Ext}{\operatorname{Ext}}

\newcommand{\Hom}{\operatorname{Hom}}

\newcommand{\sd}{\Sigma_d}

\newcommand{\HH}{\operatorname{H}}

 \numberwithin{equation}{subsection}

\begin{document}
\title[Specht module cohomology]
{\bf Cohomology and generic cohomology of Specht modules for the symmetric group}

\author{\sc David J. Hemmer}
\address
{Department of Mathematics\\ University at Buffalo, SUNY \\
244 Mathematics Building\\Buffalo, NY~14260, USA}
\thanks{Research of the  author was supported in part by NSF
grant  DMS-0556260} \email{dhemmer@math.buffalo.edu}

\date{January 2009}


\begin{abstract}

Cohomology of Specht modules for the symmetric group can be equated in low degrees with corresponding  cohomology for the Borel subgroup $B$  of the general linear group $GL_d(k)$, but this has never been exploited to prove new symmetric group results. Using work of Doty on the submodule structure of symmetric powers of the natural $GL_d(k)$-module together with work of Andersen on cohomology for $B$ and its Frobenius kernels, we prove new results about $\HH^i(\Sigma_d, S^\lambda)$. We  recover work of James in the case $i=0$. Then we prove two stability theorems, one of which is a ``generic cohomology" result for Specht modules equating cohomology of $S^{p\lambda}$ with $S^{p^2\lambda}.$  This is the first theorem we know relating Specht modules $S^\lambda$ and $S^{p\lambda}$. The second result equates cohomology of $S^\lambda$ with $S^{\lambda + p^a\mu}$ for large $a$.

\end{abstract}

\maketitle
\section{Introduction}
\label{section: Introduction}
\subsection{}
For an algebraic group $G$ and a $G$-module $M$, the Frobenius morphism on $G$ lets one define a new $G$ module  $M^{(1)}$, the Frobenius twist of $M$. If $G=GL_n(k)$ and $L(\lambda)$ is an irreducible module then $L(\lambda)^{(1)} \cong L(p\lambda)$. Thus theorems about $G$-modules which involve multiplying partitions by $p$ arise quite naturally and can  be explained in terms of the Frobenius map.

In contrast, suppose $S^\lambda$ is a Specht module for the symmetric group $\Sigma_d$, where $\lambda$ is a partition of $d$, denoted $\lambda \vdash d$. Then $p\lambda \vdash pd$, so $S^{p\lambda}$ is a Specht module for an entirely different group, $\Sigma_{pd}$,  and there is no evident relation between the two modules. Indeed we know of no theorems involving both $S^{\lambda}$ and $S^{p\lambda}$. This paper will prove the first such theorem, although an explanation involving only the symmetric group eludes us!

    I would like to thank the anonymous referee for an extensive report which vastly improved the exposition, discovered an error  in the proof of Lemma \ref{lem: E11=0} and provided a correct version.

\subsection{}

 Let $k$ be an algebraically closed field of characteristic $p\geq 3$. For a partition $\lambda \vdash d$ let $S^\lambda$ denote the corresponding Specht module for the symmetric group $\sd$, and let $S_\lambda$ be its linear dual.  In  characteristic two every Specht module is also a dual Specht module, and the problem of calculating cohomology is quite different and seems more difficult. Proposition \ref{prop: KNresultequatescohSpectwithGLd} below, which equates symmetric group cohomology with that of the general linear group, only holds through degree $2p-4$,  so does not apply even to $\HH^1(\sd, S^\lambda)$ when $p=2$. Thus we will restrict to the odd characteristic case for this paper.

 Much is known about the low degree cohomology groups $\HH^i(\sd, S_\lambda)$. In \cite{BKMdualspecht} (where  again only odd characteristic is considered), it is shown that this cohomology vanishes in degrees $1 \leq i \leq p-3.$  For $p=3$ a complete description of the nonzero cohomology was given for $i=1, 2$ in \cite[Thms. 2.4, 4.1]{BKMdualspecht}.

In contrast little is known about the cohomology $\HH^i(\Sigma_d, S^\lambda)$ of Specht modules. For $i=0$ this was computed by James  \cite[Thm. 24.4]{Jamesbook}. It is at most one-dimensional, and explicit conditions are given on $\lambda$ for it to be nonzero, see Theorem \ref{thm:JamestheoremonHom}
below.  In  \cite{HNcohomologyofspechtmodules}
the cohomology $\HH^i(\Sigma_d, S^\lambda)$ for $0 \leq i \leq p-2$ is shown to agree with certain  $B$ cohomology in degree $i+\binom{d}{2}$,
where $B$ is the Borel subgroup of $GL_d(k)$. However this high degree  $B$ cohomology has not been computed, so no new symmetric group results were obtained.

In our approach the relevant cohomology for $B$ and its Frobenius kernels can be computed, and so  new results on the cohomology $\HH^1(\Sigma_d, S^\lambda)$ of Specht modules are obtained. We prove some stability theorems that generalize the results of James for $\HH^0(\sd, S^\lambda)$ and are inspired by the known results for two-part partitions. In particular we prove a ``generic cohomology" theorem that $\HH^1(\Sigma_{p^ad}, S^{p^a\lambda})$  stabilizes for $a\geq 1$. There is a corresponding result for algebraic groups where one twists by the Frobenius automorphism. However for the symmetric group there is no Frobenius automorphism, and it is quite mysterious why any theorem should relate Specht modules $S^{\lambda}$ and $S^{p\lambda}$, which are apparently unrelated modules for two different groups!

\section{Notation and preliminaries}
\label{section: Notation and preliminaries}
\subsection{}
Although our application is to symmetric group representation theory, the actual work will be done within the general linear group theory. A basic reference for key results and also for our notation is \cite{jantzenbook2nded}. For information on the Schur algebra see \cite{Greenpolygln}. We will also draw extensively from the paper \cite{AndersenExtensionsofmodulesforalgebraicgroups} of Andersen.

Recall $k$ is an algebraically closed field of characteristic $p\geq 3$ and let $G=\operatorname{GL}_{n}(k)$ be the general linear group. Let $B$ (resp. $B_+$) be the Borel subgroup of lower (resp. upper) triangular matrices and $T$ the torus of diagonal matrices. Let $R$ denote the root system with respect to $T$, with associated inner product $\langle-, - \rangle$. Let $X(T) \cong {\mathbb Z}^n$ denote the weight lattice and $S=\{\alpha_1, \alpha_2, \ldots, \alpha_{n-1}\}$ a set of simple roots such that $B$ corresponds to the negative roots. For $\alpha \in S$ the corresponding coroot is $\alpha^\vee$ and the corresponding reflection on $X(T)$ is $s_\alpha$. Let $\rho$ denote half the sum of the positive roots.

The set of dominant weights is $X(T)_+=\{\lambda \in X(T) \mid \langle \alpha^\vee, \lambda \rangle \geq 0 \text{ for all } \alpha \in S.\}$ Let $\leq$ denote the ordering on $X(T)$ given by $\lambda \leq \mu$ if and only if $\mu - \lambda$ is a linear combination of positive roots with nonnegative coefficients. The set of $p^r$-restricted weights is denoted $X_r(T)=\{\lambda \in X(T) \mid 0 \leq \langle \alpha^\vee, \lambda \rangle < p^r \text{ for all } \alpha \in S\}$.

For a module $M$ and a simple module $S$, the composition multiplicity of $S$ in $M$ will be denoted $[M : S]$. For a $GL_n(k)$-module $M$, the module $M^{(r)}$ will denote the $r$th Frobenius twist of $M$. The $r$th Frobenius kernel of $B$ is denoted $B_r$, see \cite{jantzenbook2nded} for definitions. We will sometimes use the fact that $B/B_r$ is isomorphic to $B$ and the modules can be identified using the $r$th power of the Frobenius twist.

\subsection{}The Schur algebra $S(n,d)$ is the finite-dimensional associative $k$-algebra
$\text{End}_{k\Sigma_{d}}(V^{\otimes d})$ where $V\cong k^n$ is the
natural representation of $G$.  Excellent references for representation theory of the  Schur algebra are the books \cite{GreenPolyGLn2ndedwithappendix} and \cite{Martinbook}. The category $M(n,d)$ of polynomial $G$-modules of a fixed
degree $d\geq 0$ is equivalent to the category of modules for $S(n,d)$.

Simple $S(n,d)$-modules are in bijective correspondence with the set $\Lambda^+(n,d)$ of
partitions of $d$ with at most $n$ parts, and are denoted by $L(\lambda)$. Note
that one can also identify $\Lambda^+(n,d)$ as the set of dominant polynomial
weights of $G$ of degree $d$. For $\lambda\in \Lambda^+(n,d)$,
let $H^{0}(\lambda)=\text{ind}_{B}^{G} \lambda$ be the induced module and
let $V(\lambda)$ be the Weyl module. Often we will write $H^0(d)$ rather than $H^0((d,0,\ldots, 0))$ and similarly $L(d)$. Note we are using $\lambda$ here to denote the one-dimensional $B$ module of weight $\lambda$. For $\lambda, \mu \in \Lambda^+(n,d)$ let $\unrhd$ denote the usual dominance order \cite[Def. 1.4.2]{Martinbook}.

\subsection{}We recall  two results about $\Ext$ groups for $GL_n(k)$-modules. The first is that for two $S(n,d)$-modules $M,N$, the $\Ext^i$ groups are the same whether we work  in the category of all $G$-modules or of $S(n,d)$-modules. The second is that for $n>d$ the category of modules for $S(n,d)$ is equivalent to that of $S(d,d)$, where the equivalence  ``preserves"  $L(\lambda)$, $H^0(\lambda)$ and $V(\lambda)$ for $\lambda \vdash d$. (Although the modules certainly have different dimensions, for example $V(1,1,1)$ is a one-dimensional $S(3,3)$ module but a four-dimensional $S(4,3)$ module.) Both results are special cases of a more general result regarding truncated categories (cf. \cite[Chpt. A]{jantzenbook2nded}).

\begin{prop}\cite[A.10, A.18]{jantzenbook2nded}
\label{prop: extagrees  in S or G category} Let $M$ and $N$  be $S(n,d)$-modules, and hence also $G$-modules. Then for all $i \geq 0$
$$\Ext^i_{S(n,d)}(M,N) \cong \Ext^i_G(M,N).$$
\end{prop}
The next proposition will let us equate $\sd$ cohomology with that of $GL_n(k)$ for various $n \geq d$.

\begin{prop}
 \label{prop: ext agrees GLnGLd}
 \cite[6.5g]{Greenpolygln} Suppose $n>d$. Then there is an idempotent $e \in S(n,d)$ such that $eS(n,d)e \cong S(d,d)$ and the functor taking $M$ to $eM$ is an equivalence of categories from mod-$S(n,d)$ to mod-$S(d,d)$ mapping $L(\lambda)$ to $L(\lambda)$ and similarly for $H^0(\lambda)$ and $V(\lambda)$.
\end{prop}

\section{$B$-Cohomology and Spectral Sequences}
\label{sec: B-cohomologyandspectralsequences}
\subsection{}
As we will see in the next section, low degree Specht module cohomology can be equated with certain $B$ cohomology, and it is in the setting of $B$ cohomology that our work will be done.
In this section we collect the results on $B$ and $B_r$ cohomology that we will need, and recall two important spectral sequences that we will use. The first describes extensions between $B$ modules that have been twisted. The second relates cohomology for $B$ and $B_r$.

First we recall that $L(\lambda)$ has simple head as a $B$-module and, if $\lambda \in X_r(T)$,  also as a $B_r$-module.

\begin{prop}
\label{prop: HeadofLlambdaasBandBrmodule}
Let $\lambda \in X(T)_+$ and $\nu \in X(T)$. Then:
\begin{itemize}
  \item[(a)] $\Hom_B(L(\lambda), \nu) \cong$ $\begin{cases}  k&\text{if } \lambda=\nu\\ 0 & \text{otherwise.} \end{cases}$

      \item[(b)]
          $\Hom_{B_r}(L(\lambda), \nu) \cong$ $\begin{cases}  p^r\nu_1 &\text{if } \nu=\lambda+p^r\nu_1\\ 0 & \text{otherwise.} \end{cases}$ where the isomorphism is as $B$-modules.
\end{itemize}
\end{prop}

\begin{proof}
Part (a) follows from Frobenius reciprocity (Proposition \ref{prop: frobenius reciprocityforH0lambda}(a) below) and the fact
\cite[II.2.3]{jantzenbook2nded} that $H^0(\lambda)$ has simple socle $L(\lambda)$. Part (b) is Equation 3.1 in \cite{AndersenExtensionsofmodulesforalgebraicgroups}.
\end{proof}

Calculating higher degree cohomology for $B$ or $B_r$ is a difficult problem and a subject of active research. Fortunately we need only a few results in degree one. The first part of the following is immediate from Lemma 2.2 in \cite{CPSvDK}, where a much stronger result is proven. We denote by $St_r$ the $r$th Steinberg module $L((p^r-1)\rho)$.

\begin{prop}
\label{prop: degree1Bcoho}
Let $\lambda \in X(T)_+$ and $\nu \in X(T)$. Then:

\begin{itemize}
\item[(a)]If $\lambda \ngeq \mu$ then $\Ext^1_B(L(\lambda), \mu)=0$.
\item[(b)] \cite[II.12.1]{jantzenbook2nded} $\HH^1(B_r, k)=0.$
    \item[(c)] \cite[Prop. 3.2]{AndersenExtensionsofmodulesforalgebraicgroups} Suppose further that $\lambda, \nu \in X_r(T).$ Then the weights of $\Ext^1_{B_r}(L(\lambda), \nu)$, as a $B$-module, are contained in the set
        $$\{p^r \xi \mid \lambda - \nu+p^r\xi - (p^r-1)\rho \text{ is a weight in } St_r.\}$$
\end{itemize}
\end{prop}

\subsection{}We will make use of two spectral sequences, both of which are first quadrant spectral sequences of the form $E_2^{*,*}$ converging to $\HH^*$.

 The first is a spectral sequence that we will apply to compare $\Ext_B^*(M_1, M_2)$ with  $\Ext_B^*(M_1^{(r)}, M_2^{(r)})$.
\begin{prop}\cite[I.6.10, II.10.14]{jantzenbook2nded}
\label{prop:SSrelatingextFrobenius}
Let $M_1, M_2$ be $B$-modules and $B_r$ be the $r$th Frobenius kernel. There is a first quadrant spectral sequence:
\begin{equation}
\label{eq: spectralsequencetwistingcoho}
E_2^{i,j}=\Ext^i_B(M_1, M_2 \otimes \HH^j(B_r, k)^{(-r)}) \Rightarrow \Ext^{i+j}_B(M_1^{(r)}, M_2^{(r)}).
\end{equation}
\end{prop}

The second is the Lyndon-Hochschild-Serre spectral sequence.

\begin{prop}\cite[I.6.6(1)]{jantzenbook2nded} Let $H$ denote either $G$ or $B$. Let $E$ and $V$ be $H$-modules and $M$ be an $H/H_r$-module. Then there is a first quadrant spectral sequence:

\begin{equation}
\label{eq: LHSspectralsequence}
E_2^{i,j}=\Ext^i_{H/H_r}(M, \Ext^j_{H_r}(E,V)) \Rightarrow \Ext^{i+j}_H(M \otimes E, V).
\end{equation}
\end{prop}

 We will make use only of a few elementary properties, described below,  of these spectral sequences, so the reader does not need extensive familiarity with them. Specifically we need:

\begin{prop}
\label{prop: Spectralsequencesummary}
 Suppose $E_2^{*,*}$ converging to $H^*$ is one \eqref{eq: spectralsequencetwistingcoho} or \eqref{eq: LHSspectralsequence}.

\begin{itemize}
\item[(a)]There is a five-term exact sequence:
\begin{equation}
\label{eq: 5term}
0 \rightarrow E_2^{1,0} \rightarrow H^1 \rightarrow E_2^{0,1} \rightarrow E_2^{2,0} \rightarrow H^2.
\end{equation}

\item[(b)]Suppose $E_2^{i,j}=0$ for all $i+j=n$. Then $H^n=0$.

\item[(c)] $E_2^{0,0} \cong H^0.$
\end{itemize}
\end{prop}

\section{Specht module cohomology and symmetric powers}
\label{section: Specht module cohomology and symmetric powers}

\subsection{}
Although we will be studying symmetric group cohomology, we  first transfer the problem to the general linear group. We  recall below the result that allows  this. For additional such results and explanations of how they are obtained using the Schur functor and the higher derived functors of its adjoint, see \cite{KNcomparingcohom}. We use only  the following.

\begin{prop}\cite[Corollary 6.3(b)(iii)]{KNcomparingcohom} For $0 \leq i \leq 2p-4$  we have
$$\HH^i(\sd, S^\lambda) \cong \Ext^i_{GL_d(k)}(H^0(d), H^0(\lambda)).$$
\label{prop: KNresultequatescohSpectwithGLd}
\end{prop}
\begin{rem}Proposition \ref{prop: KNresultequatescohSpectwithGLd} is stated in \cite{KNcomparingcohom} for $0 \leq i \leq 2p-3$ but actually only holds up to $i=2p-4$, cf. \cite[2.4]{HNcohomologyofspechtmodules} for an explanation.
\end{rem}

We will actually do our computations using $B$ cohomology, so the following is crucial, where part (b) is immediate from part (a) and Propositions \ref{prop: ext agrees GLnGLd} and \ref{prop: KNresultequatescohSpectwithGLd}.

\begin{prop}
\label{prop: frobenius reciprocityforH0lambda}

\begin{itemize}
\item[]
\item[(a)] \cite[4.7a]{jantzenbook2nded}
Let $V$ be a $G$-module and $\lambda \vdash d$. Then for all $i \geq 0$:
$$\Ext^i_G(V,H^0(\lambda)) \cong \Ext^i_B(V,\lambda).$$

\item[(b)] If $0 \leq i \leq 2p-4$ and $G=GL_n(k)$ for $n \geq d$ then
$$\HH^i(\sd, S^\lambda) \cong \Ext^i_B(H^0(d), \lambda).$$
\end{itemize}

\end{prop}

\subsection{Structure of symmetric powers} Proposition \ref{prop: frobenius reciprocityforH0lambda}(b) suggests one must understand $H^0(d)$ in order to compute low degree Specht module cohomology. The module $H^0(d)$ is isomorphic to the $d$th symmetric power of the natural module $V \cong k^n$. In \cite{DotysubmoduelstructureWeylmodules}, Doty gave a complete description of the composition factors and submodule lattice of $H^0(d)$, which we recall below.

Let
$$B(d):=\{\beta=(\beta_1, \beta_2, \ldots, \beta_ n) \mid \beta_i \geq 0 \,\,\mathrm{ and } \sum \beta_i=d\}.$$ Then $H^0(d)$ has a basis of weight vectors indexed naturally by $B(d)$ \cite[2.1]{DotysubmoduelstructureWeylmodules}.
For $\beta \in B(d)$, associate a sequence of
nonnegative integers $c_i(\beta)$ as follows. First write each
$\beta_i$ out in base $p$. Then add them base $p$ to get $d$. For $i\geq
1$, let $c_i(\beta)$ be the number that is ``carried" to the top of
the $p^i$ column during the addition. For example if $p=3$ and
$\beta=(5,5,2)$ then the addition $5+5+2$=12 base three looks like:
$$\begin{array}{ccc}
  _{1} &_{2} &  \\
   & 1 & 2 \\
   & 1 & 2 \\
   +&0& 2 \\\hline
  1 & 1 & 0 \\
\end{array}$$
so $c(\beta)=(2,1)$. Doty calls this the {\it carry pattern}
of $\beta$.

Let $C(d)$ be the set of all carry patterns which occur for some $\beta \in B(d)$. Define a partial order on $C(d)$ by declaring $(c_1, c_2, \ldots, c_m) \leq (c_1', c_2', \ldots, c_m')$ if $c_i \leq c_i'$ for all $i$. Given a subset $B\subseteq C(d)$, let $T_B$ be the subspace of $H^0(d)$ spanned by all $\beta$ with $c(\beta) \in B$. We can now state Doty's result:

\begin{thm}\cite{DotySubmodulesofsymmetricpowers}
\label{thm:Dotystructureofsymmetricpowers}
The correspondence $B \rightarrow T_B$ defines a lattice isomorphism between the lattice of ideals in the poset $C(d)$ and the lattice of submodules of $H^0(d)$. In particular the composition factors of $H^0(d)$ are in one-to-one correspondence with the subspaces $T_c$, for $c \in C(d)$. If $L(\lambda)$ corresponds to $T_c$ then $\lambda$ is the maximal partition in $\Lambda^+(n,d)$ with carry pattern $c$.
\end{thm}

Given $\lambda \in \Lambda^+(n,d)$, it is not difficult to determine if $[H^0(d): L(\lambda)] \neq 0$. The next lemma is an easy exercise working with carry patterns.

\begin{lem}
\label{lem: tells youcompositionfactorsofsympowersinpadicexpansion} Let $\lambda =(\lambda_1, \lambda_2, \ldots, \lambda_n)\in \Lambda^+(n,d)$. Define numbers $a_{ij}$ with $0 \leq a_{ij}<p$ by
$\lambda_i=\sum a_{ij}p^j.$ Then:
\begin{itemize}
\item[(a)] $[H^0(d): L(\lambda)] \neq 0$ if and only if $a_{ij}\neq p-1$ implies $a_{lj}=0$ for all $l>i.$

\item[(b)] $[H^0(pd):L(p\lambda)]=[H^0(d): L(\lambda)].$
\end{itemize}
\end{lem}

We will later use a strengthened version of Lemma \ref{lem: tells youcompositionfactorsofsympowersinpadicexpansion}(b), namely that $H^0(d)^{(1)}$ is a submodule of $H^0(pd)$ in a particularly nice way, cf. Proposition \ref{prop: H0twistedintoH0pd}.

\section{Previously known Specht module cohomology}
\subsection{Computing $\Hom_{\sd}(k,S^\lambda)$}

The degree zero cohomology $\HH^0(\Sigma_d, S^\lambda) \cong \Hom_{\sd}(k, S^\lambda)$ was determined by James. For an integer $t$ let $l_p(t)$ be the least nonnegative integer
satisfying $t<p\,^{l_p(t)}$. James proved:

\begin{thm}\cite[24.4]{Jamesbook}
\label{thm:JamestheoremonHom}
The cohomology, $\HH^0(\Sigma_d, S^\lambda)$ is zero unless  $\lambda_i \equiv -1 \mbox{ \rm mod }
p\,^{l_p(\lambda_{i+1})}$ for all $i$, in which case it is one-dimensional.
\end{thm}

\begin{rem}
\label{rem: James result comesfromDoty} James' result is proved entirely using  symmetric group theory, applying the famous kernel intersection theorem. We observe that Doty's description of $H^0(d)$  allows one to calculate precisely  when $\Hom_B(H^0(d), \lambda)$ is nonzero, and thus rederive James' result in an entirely different way. The space of homomorphisms can only be nonzero for $\lambda \vdash d$ with $[H^0(d):L(\lambda)]=1$, so we must determine for each  $L(\lambda)$ a constituent of $H^0(d)$, if there is a nonzero $B$ module homomorphism from $H^0(d)$ to $\lambda$.  There is such a map precisely when there does not exist a $\mu \rhd \lambda$ such that $[H^0(d) : L(\mu)] =1$ and $c(\mu) > c(\lambda)$. This can easily be seen to be equivalent to the condition in Theorem \ref{thm:JamestheoremonHom}. Since the result is already known, we leave the details to the reader. One must show that given such $\lambda$ and $\mu$ (i.e. $\mu \rhd \lambda$ and $c(\mu) >c(\lambda$))   the $\lambda$ weight vector always lies in the $B$ submodule generated by the $\mu$ weight vector, so there is no homomorphism in this case. Doty's paper explicitly describes the $B$ action, so this can be done.
\end{rem}

\begin{exm}In characteristic five, $[H^0(25): L(20,5)] =1$ but $\HH^0(\Sigma_{25}, S^{(20,5)})=0$. This is because $(24,1) \rhd (20,5)$ and $c(24,1)=(1,1) > (1,0)= c(20,5).$ The $(20,5)$ weight vector lies in the $B$ submodule generated by the $(24,1)$ weight vector, and so $\Hom_B(H^0(25), (20,5))=0.$
\end{exm}

\subsection{Two-part partitions}
The other known example of $\HH^1(\Sigma_d, S^{\lambda})$ is for $\lambda=(\lambda_1, \lambda_2)$ a two-part partition, where the answer can be deduced from work of Erdmann \cite{ErdmannExt1WeylmodulesSL2} on $\Ext^1$ between Weyl modules for $SL_2(k)$, as we describe below.

For $r>0$ write $r=\sum_{i \geq 0}r_ip^i$ with $0 \leq r_i \leq p-1$. For $0 \leq a \leq p-1$ define $\overline{a}$ so that $0 \leq \overline{a} \leq p-1$ and $a + \overline{a} \equiv p-1$ (mod $p$). Erdmann defined \cite[p. 456]{ErdmannExt1WeylmodulesSL2} a collection of integers $\Psi_p(r)$ by
\begin{equation}
\label{eq: definitionofPsi(r)}
\Psi_p(r)=\{\sum_{i=0}^{u-1}\overline{r_i}p^i +p^{u+a} : r_u \neq p-1, a \geq 1, u \geq 0\} \cup \{\sum_{i=0}^{u}\overline{r_i}p^i : r_u \neq p-1, u \geq 0\}.
\end{equation}

From Lemma 3.2.1 in \cite{HNcohomologyofspechtmodules} we deduce $\HH^1(\Sigma_d, S^{(\lambda_1, \lambda_2)})$ is at most one-dimensional, and it is nonzero precisely when $\lambda_2 \in \Psi_p(\lambda_1-\lambda_2)$.  The conditions on $\lambda$ for this to occur were stated incorrectly in \cite{HNcohomologyofspechtmodules}, so we correct it here.

First suppose $\lambda_2$ lies in the second set in the description \eqref{eq: definitionofPsi(r)} of $\Psi_p(\lambda_1-\lambda_2)$. So we have:

\begin{eqnarray}
\label{eq: lambda1lambda2}
  \lambda_1 - \lambda_2 &=& r_0+r_1p+r_2p^2 + \cdots \\\nonumber
  \lambda_2 &=&  \overline{r_0}+ \overline{r_1}p + \cdots + \overline{r_u}p^u \\\nonumber
  \lambda_1&=&p-1+(p-1)p+(p-1)p^2+ \cdots + (p-1)p^u+ r_{u+1}p^{u+1} + \cdots
\end{eqnarray}
where $\overline{r_u} \neq 0$.

Recall from Theorem \ref{thm:JamestheoremonHom}  that $\HH^0(\Sigma_d, S^{(\lambda_1, \lambda_2)}) \neq 0$, precisely when $\lambda_1 \equiv -1$ mod $p^{l_p(\lambda_2)}$. Equation \ref{eq: lambda1lambda2} demonstrates that $l_p(\lambda_2)=u+1$ and $\lambda_1 \equiv -1$ mod $p^{u+1}$. Thus we have the exactly James' criterion and obtain Proposition \ref{prop:twopartpartitions}(a) below.

Next we consider when $\lambda_2$ lies in the first subset in \eqref{eq: definitionofPsi(r)}. In this case we will have, for some $u \geq 0$ and $a \geq 1$:

\begin{eqnarray}
\label{eq: lambda1lambda2case2}
  \lambda_1 - \lambda_2 &=& r_0+r_1p+r_2p^2 + \cdots +r_{u-1}p^{u-1}+r_up^u+\cdots\\\nonumber
  \lambda_2 &=&  \overline{r_0}+ \overline{r_1}p + \cdots + \overline{r_{u-1}}p^{u-1}+0p^u+p^{u+a} \\\nonumber
  \lambda_1&=&p-1+(p-1)p+(p-1)p^2+ \cdots + (p-1)p^{u-1}+ r_{u}p^{u} + \cdots
\end{eqnarray}
where ${r_u} \neq p-1.$ In \eqref{eq: lambda1lambda2case2}, $u$ is minimal such that $\lambda_1 \equiv -1$ mod $p^u$ but not $\equiv -1$ mod $p^{u+1}$. This proves the second part of:

\begin{prop}
\label{prop:twopartpartitions}
Let $\lambda=(\lambda_1, \lambda_2) \vdash d$ with $\lambda_2 \neq 0$. Then
$H^1(\Sigma_d,S^\lambda)$ is zero except in the two cases below, when it is one-dimensional:

\begin{itemize}
\item[(i)] If $\Hom_{\Sigma_d}(k,S^\lambda) \neq 0$ or
\item[(ii)] If  $\lambda_1 \equiv -1$ mod $p^u$ but $\lambda_1 \not \equiv -1$ mod $p^{u+1}$ for some $u \geq 0$ such that $\lambda_2=c+p^b$ for $c<p^u$ and $b>u$.
\end{itemize}
\end{prop}

\begin{exm} In characteristic five, $\HH^1(\Sigma_{54}, S^{(29,25)}) \cong k$. Here choose $u=1, a=2, c=0$ in part (ii) above.
\end{exm}

\begin{exm}
\label{exm: papb}
In  characteristic $p$ suppose $a,b >0$ and $pa \geq p^b$. Then choosing $u=0$ we get $ \HH^1(\Sigma_{d}, S^{(pa, p^b)}) \neq 0$.
\end{exm}

Case $(i)$ in Proposition \ref{prop:twopartpartitions} is a special case of a more general result conjectured in \cite{HNcohomologyofspechtmodules}, which can be proved using a suggestion of Andersen.

\begin{prop}[Andersen]
\label{prop: Homneq0impliesExt1neq0}
Suppose $\HH^0(\Sigma_d, S^\lambda)\neq 0$ and $\lambda \neq (d)$. Then $\HH^1(\Sigma_d, S^\lambda) \neq 0$.
\end{prop}
\begin{proof}
This follows from the universal coefficient theorem. The key observation is that the Specht module is defined over the integers, but when $\lambda \neq (d)$ then $\Hom_{\mathbb{Z}}({\mathbb Z}, S^\lambda_{\mathbb Z})=0$.
\end{proof}

\subsection{}

 The following corollary is immediate from Proposition \ref{prop:twopartpartitions}, and is the motivation for the generalizations proved in the next two sections.
\begin{cor}
\label{cor:2partmotivation} Let $\lambda=(\lambda_1, \lambda_2) \vdash d<p^a$. Then
\begin{itemize}
\item[(a)] $\HH^1(\Sigma_{pd}, S^{p\lambda}) \cong \HH^1(\Sigma_{p^2d}, S^{p^2\lambda}).$
\item[(b)] $\HH^1(\Sigma_d, S^\lambda) \cong \HH^1(\Sigma_{d+p^a}, S^{(\lambda_1+p^a, \lambda_2)})$

\end{itemize}
\end{cor}

\section{Generic cohomology for Specht modules?}
\label{section: Generic cohomology for Specht modules}
\subsection{}
For any $G$-module $M$ there is a series of injections:

\begin{equation}
\label{eq: injectionsingenericcohomology}
\HH^i(G,M) \rightarrow \HH^i(G, M^{(1)}) \rightarrow \HH^i(G, M^{(2)}) \rightarrow \cdots.
\end{equation}
This sequence is known \cite{CPSvDK} to stabilize, and the limit is called the generic cohomology of $M$. For example we get injections $\HH^i(G,L(\lambda)) \rightarrow \HH^i(G, L(p\lambda)) \rightarrow \HH^i(G, L(p^2\lambda))$. It is  natural then to have theorems for $G$ which involve multiplying a partition by $p$, as this reflects what happens to the weights after a Frobenius twist.

For the symmetric group, we know of no theorems involving multiplying a partition by $p$. There is nothing that seems to play the role of the Frobenius twist. Moreover, the modules $S^\lambda$ and $S^{p\lambda}$ are not even modules for the same symmetric group. However, for two-part partitions we observed in  Corollary \ref{cor:2partmotivation} that $\HH^1(\Sigma_{pd}, S^{p\mu}) \cong \HH^1(\Sigma_{p^2d}, S^{p^2\mu})$. In this section we generalize this stability result to arbitrary partitions, and show there is a generic cohomology for Specht modules in degree one. The main result is an isomorphism $\HH^1(\Sigma_{pd}, S^{p\lambda}) \cong \HH^1(\Sigma_{p^2d}, S^{p^2\lambda})$.

\subsection{Relating cohomology and Frobenius twists}


We can use the spectral sequence \eqref{eq: spectralsequencetwistingcoho} in our first step towards relating cohomology of $S^\lambda$ and $S^{p\lambda}$.


\begin{lem}
\label{lem: Bcohooftwsit agrees} Let $\mu \vdash d$. Then
$$\Ext^1_B(H^0(d)^{(1)}, p\mu) \cong \Ext^1_B(H^0(d), \mu).$$
\end{lem}
\begin{proof}
Take $r=1$ and consider the five-term exact sequence \eqref{eq: 5term} for the spectral sequence \eqref{eq: spectralsequencetwistingcoho}. Choosing $M_1=H^0(d)$ and $M_2=\mu$, we obtain:
$$0 \rightarrow \Ext^1_B(H^0(d), \mu) \rightarrow \Ext^1_B(H^0(d)^{(1)}, p\mu) \rightarrow \Hom_B(H^0(d), H^1(B_1, k)^{(-1)} \otimes \mu) \rightarrow \cdots.$$ Recall from Proposition \ref{prop: degree1Bcoho}(b) that $H^1(B_1, k)=0$, so we have the desired isomorphism.
\end{proof}
Lemma \ref{lem: Bcohooftwsit agrees} can be interpreted as the immediate stabilization of \eqref{eq: injectionsingenericcohomology} when $M=\mu \otimes H^0(d)^*$.

 \begin{rem} If the left side in Lemma \ref{lem: Bcohooftwsit agrees} had $H^0(pd)$ instead of $H^0(d)^{(1)}$ we would have an isomorphism between $\HH^1(\Sigma_d, S^\lambda)$ and $\HH^1(\Sigma_{pd}, S^{p\lambda})$, however this is false in general.
 \end{rem}

Next we prove a technical lemma:

\begin{lem}
\label{lem: weightsbig}
Suppose $\lambda \vdash p^2d$ such that $\lambda \neq p\tau$ for any $\tau \vdash pd$. Suppose further that $[H^0(p^2d) : L(\lambda)] \neq 0$ and that  $\mu \vdash d$ with $\lambda \rhd p^2\mu$. Then:

$$\langle \lambda-p^2\mu, \alpha_i^\vee \rangle \geq p^2$$ for some $i$.
\end{lem}
\begin{proof} Let $\lambda=(\lambda_1, \lambda_2, \ldots, \lambda_n)$ and denote the $p$-adic expansion of $\lambda_i$ by
$\lambda_i=a_{i,0} + a_{i,1}p+ \cdots.$ Since $p \nmid \lambda_1$ by Lemma \ref{lem: tells youcompositionfactorsofsympowersinpadicexpansion}(a), we must have $\lambda_1 >p^2\mu_1$. Since $\lambda$ and $p^2\mu$ both partition $p^2d$, there must be some $i \geq 1$ with $\lambda_i \geq p^2\mu_i$ and $\lambda_{i+1}<p^2\mu_{i+1}$. So let:
\begin{eqnarray*}
  \lambda_i &=& a_{i,0} + a_{i,1}p+p^2t \\
    \lambda_{i+1}&=&a_{i+1,0}+a_{i+1,1}p+p^2s
\end{eqnarray*}
Our assumptions on $\lambda$ and $p^2\mu$ imply that $t \geq p^2\mu_i$ and $s< p^2\mu_{i+1}$. By Lemma \ref{lem: tells youcompositionfactorsofsympowersinpadicexpansion}(a) we have $a_{i,l} \geq a_{i+1,l}$ for all $l$. Thus:

\begin{eqnarray*}
\langle \lambda-p^2\mu, \alpha_i^\vee \rangle &=& (\lambda_i-p^2\mu_i)-(\lambda_{i+1}-p^2\mu_{i+1})\\
&=&a_{i,0}+a_{i,1}p+p^2(t-p^2\mu_i)-a_{i+1,0}-a_{i+1,1}p-p^2(s- p^2\mu_{i+1})\\
& \geq & p^2(t-p^2\mu_i)-p^2(s- p^2\mu_{i+1})\\ &\geq& p^2\\
\end{eqnarray*}\end{proof}

\begin{cor}
\label{cor:weightnotinsteinberg}
Let $\lambda$ and $ p^2\mu$ be as in Lemma \ref{lem: weightsbig} above. Then $\lambda-p^2\mu - (p-1)\rho$ is not a weight in the Steinberg module $St_1=L((p-1)\rho)$.
\end{cor}

\begin{proof}
Let $\gamma= \lambda-p^2\mu - (p-1)\rho$ and choose $i$ as in Lemma \ref{lem: weightsbig}. Then
\begin{eqnarray*}
  \langle (p-1)\rho-\gamma, \alpha_i^\vee\rangle&=& \langle2(p-1)\rho - (\lambda-p^2\mu), \alpha_i^\vee\rangle \\
  &=& 2p-2- \langle (\lambda-p^2\mu), \alpha_i^\vee\rangle \\
   &\leq& 2p-2-p^2 {\text{ by Lemma } } \ref{lem: weightsbig}\\
   &=& -p^2+2p-2<0. \\
\end{eqnarray*}
 Thus $(p-1)\rho \not\geq \gamma$ so $\gamma$ is not a weight in $St_1$.
\end{proof}
\begin{rem}
\label{rem: whypdoesnotwork}
We observe that Corollary \ref{cor:weightnotinsteinberg} requires $p^2\mu$ and the corresponding statement for $p\mu$ is false. For example if $p=5$, $\lambda=(9,1)$ and $5\mu=(5,5)$ then $\lambda-5\mu-4\rho$ is a weight in the Steinberg module. This is why our stability theorem requires  comparing $p\mu$ and $p^2\mu$.
\end{rem}

\subsection{} Corollary \ref{cor:weightnotinsteinberg} lets us prove a key vanishing result for $B$ cohomology, which we state next.

\begin{prop}
\label{prop:vanishingext1B}
Let $\lambda \vdash p^2d$ with $[H^0(p^2d):L(\lambda)] \neq 0$. Suppose $\lambda$ is not of the form $p\tau$ and let $\mu \vdash d$. Then:
$$\Ext^1_B(L(\lambda), p^2\mu)=0.$$
\end{prop}

\begin{proof} We can assume that $\lambda >p^2\mu$ without loss by Proposition \ref{prop: degree1Bcoho}(a). Write $\lambda=\lambda_{(0)}+p\tau$ with $0 \neq \lambda_{(0)} \in X_1(T)$ and, using the Steinberg tensor product theorem, we have:

$$\Ext^1_B(L(\lambda), p^2\mu) \cong \Ext^1_B(L(\lambda_{(0)}) \otimes L(p\tau) \otimes (-p^2\mu), k).$$ Now consider the spectral sequence \eqref{eq: LHSspectralsequence}  with $B_1 \lhd B$. Set $M=L(p\tau) \otimes -p^2\mu$, $E=L(\lambda_{(0)})$ and $V=k$ to obtain:

\begin{equation}
\label{eq:LHSB1inB}
E_2^{i,j}=\Ext^i_{B/B_1}\left( L(p\tau) \otimes (-p^2\mu), \Ext^j_{B_1}(L(\lambda_{(0)}), k)\right) \Rightarrow \Ext^{i+j}_B(L(\lambda), p^2\mu).
\end{equation}

By Proposition \ref{prop: HeadofLlambdaasBandBrmodule}(b), we know  $\Hom_{B_1}(L(\lambda_{(0)}), k)=0$, and thus the $E_2^{1,0}$ term in \eqref{eq:LHSB1inB} is zero. Now consider the $E_2^{0,1}$ term. For any $\xi \in X(T)$ we have:

$$\Hom_{B/B_1}(L(\tau)^{(1)} \otimes (-p^2\mu), p\xi) \cong \Hom_B(L(\tau) \otimes (-p\mu), \xi) \cong \Hom_B(L(\tau), \xi+p\mu)$$
which, by Proposition \ref{prop: HeadofLlambdaasBandBrmodule}(a), is zero unless $\tau=\xi+p\mu$.

Applying this to the $E_2^{0,1}$ term in \eqref{eq:LHSB1inB}, we see that $E_2^{0,1}$ is zero unless $p\xi=p\tau -p^2\mu$ is a weight in $\Ext^1_{B_1}(L(\lambda_{(0)}), k)$. By Proposition \ref{prop: degree1Bcoho}(c), this can only occur if $\lambda_{(0)} +p\tau-p^2\mu-(p-1)\rho$ is a weight in $St_1$. But this is ruled out by Corollary \ref{cor:weightnotinsteinberg}. Thus the $E_2^{0,1}$ term vanishes as well, and so $\Ext^1_B(L(\lambda), p^2\mu)=0$ by Proposition \ref{prop: Spectralsequencesummary}(b).

\end{proof}

\subsection{Twisted symmetric powers} Our next observation is that the twisted symmetric power $H^0(d)^{(1)}$ embeds nicely in $H^0(pd)$ with cokernel containing no simple modules of the form $L(p\tau)$.

\begin{prop}
\label{prop: H0twistedintoH0pd} There is a short exact sequence
\begin{equation}
    \label{eq: SESH0dintoH0pd}
    0 \rightarrow H^0(d)^{(1)} \rightarrow H^0(pd) \rightarrow Q \rightarrow 0
\end{equation}
where for all $\tau \vdash d$, $[Q: L(p\tau)]=0$.
\end{prop}
\begin{proof} Apply $\Hom_G(-, H^0(pd))$ to the sequence $0 \rightarrow L(pd) \rightarrow H^0(d)^{(1)} \rightarrow U \rightarrow 0$ to obtain:
$$ 0 \rightarrow \Hom_G(H^0(d)^{(1)}, H^0(pd)) \rightarrow k \rightarrow \Ext^1_G(U, H^0(pd)).$$ But $\Ext^1_G(U, H^0(pd))=0$ since $(pd)$ is not dominated by any weight in $U$ (see \cite[II.4.14]{jantzenbook2nded}). Thus $ \Hom_G(H^0(d)^{(1)}, H^0(pd)) \cong k$ and by comparing socles we see the map must be an injection. The statement about $Q$ is immediate  by Lemma \ref{lem: tells youcompositionfactorsofsympowersinpadicexpansion}(b) since $H^0(pd)$ is multiplicity free.
\end{proof}

Now consider  \eqref{eq: SESH0dintoH0pd} with $pd$ and $p^2d$, and suppose $p^2\mu \vdash p^2d$. Applying $\Hom_B(-, p^2\mu)$ to \eqref{eq: SESH0dintoH0pd} we obtain:

\begin{equation}
    \label{eq: LESBcoho}
     \cdots \rightarrow \Ext^1_B(Q, p^2\mu) \rightarrow \Ext^1_B(H^0(p^2d), p^2\mu) \rightarrow \Ext^1_B(H^0(pd)^{(1)},p^2\mu) \rightarrow \Ext^2_B(Q, p^2\mu)\rightarrow \cdots
\end{equation}

We will show, in Lemma \ref{lem:B1cohovanish} and Proposition  \ref{prop:ext2withQ=0} below that the first and last term in \eqref{eq: LESBcoho} are zero.

\begin{lem}
\label{lem:B1cohovanish}
Let $Q$ be as in \eqref{eq: LESBcoho}. Then $\Ext^1_B(Q, p^2\mu)=0$.
\end{lem}
\begin{proof}
Let $[Q: L(\lambda)]\neq 0$. By Proposition \ref{prop: H0twistedintoH0pd} we know $\lambda$ is not of the form $p\tau$. The result then follows by Proposition \ref{prop:vanishingext1B}.

\end{proof}

\subsection{\textbf{Analyzing $\Ext^2_B(Q, p^2\mu)$}}
 Next we prove the analogue of Proposition \ref{prop:vanishingext1B} for $\Ext^2_B(L(\lambda), p^2\mu)$, which will require even more intricate spectral sequence calculations. Let $\lambda = \lambda_{(0)} + p\tau$ with $0 \neq \lambda_{(0)} \in X_1(T)$ as before and consider the spectral sequence \eqref{eq:LHSB1inB}. We will prove $\Ext^2_B(L(\lambda), p^2\mu)$  is zero by showing the terms $E_2^{0,2}, E_2^{1,1}$ and $E_2^{2,0}$ all vanish, and applying Proposition \ref{prop: Spectralsequencesummary}(b).

\begin{lem}
\label{lem: E20=0}
The $E_2^{2,0}$ term in \eqref{eq:LHSB1inB} is zero.
\end{lem}
\begin{proof} This is immediate since $\lambda_{(0)} \neq 0$  implies $\Hom_{B_1}(L(\lambda_{(0)}), k)=0$ by Proposition \ref{prop: HeadofLlambdaasBandBrmodule}(b).
\end{proof}

To show the $E_2^{1,1}$ term is zero we extend an argument of Andersen's. For $\epsilon \in X_1(T)$ he defined (\cite[p.495]{AndersenExtensionsofmodulesforalgebraicgroups}) $R(\epsilon)$ by the exact sequence

\begin{equation}
\label{eq:andersenexactseqdefinignR}
0 \rightarrow \epsilon \rightarrow St_1 \otimes [(p-1)\rho + \epsilon] \rightarrow R(\epsilon) \rightarrow 0.
\end{equation}
We will use this to prove:

\begin{lem}
\label{lem: E02=0}
The $E_2^{0,2}$ term in \eqref{eq:LHSB1inB} is zero.
\end{lem}
\begin{proof}
The $E_2^{0,2}$ term is:
\begin{equation}
\label{eq: e}
\Hom_{B/B_1}\left(L(p\tau) \otimes (-p^2\mu), \Ext^2_{B_1}(L(\lambda_{(0)}), k)\right).
\end{equation}
We will show that $\Ext^2_{B_1}(L(\lambda_{(0)}), k)=0$ cannot have $p\tau-p^2\mu$ as a weight.

Apply $\Hom_{B_1}(L(\lambda_{(0)}), -)$ to \eqref{eq:andersenexactseqdefinignR} with $\epsilon=0$. Using the fact that $St_1$ is injective we obtain:
$$\Ext^2_{B_1}(L(\lambda_{(0)}), k) \cong \Ext^1_{B_1}(L(\lambda_{(0)}), R(0)).$$ Suppose $\sigma$ is a weight in $R(0)$, so $\nu:=\sigma-(p-1)\rho \in St_1$. If
$\Ext^1_{B_1}(L(\lambda_{(0)}), \sigma)$ has a weight $p\tau-p^2\mu$ in it, then by  Proposition \ref{prop: degree1Bcoho}(c) we must have
$$\lambda_{(0)} - \sigma+p\tau-p^2\mu-(p-1)\rho \in St_1.$$ Then adding $\nu\in St_1$ we must obtain a weight in $St_1 \otimes St_1$. Thus
$$\omega:=\lambda-p^2\mu -2(p-1)\rho \in St_1 \otimes St_1.$$ The module $St_1 \otimes St_1$  has highest weight $2(p-1)\rho$. But choosing $i$ as in Lemma \ref{lem: weightsbig}, we have:
\begin{eqnarray*}
  \langle 2(p-1)\rho -\omega, \alpha_i^\vee \rangle&=&\langle 4(p-1)\rho-(\lambda-p^2\mu), \alpha_i^\vee\rangle\\
  &=& 4p-4 - \langle \lambda -p^2\mu, \alpha_i^\vee \rangle \\
   &\leq& 4p-4-p^2 \\
   &=& -p^2+4p-4 < 0 .
\end{eqnarray*}
Thus $\omega$ can not be a weight in $St_1 \otimes St_1$, so $p\tau-p^2\mu$ is not a weight in $\Ext^1_{B_1}(L(\lambda_{(0)}),k)$, which shows that \eqref{eq: e} is zero, as desired.
\end{proof}

Before tackling the $E_2^{1,1}$ term we prove a preliminary lemma:

\begin{lem}
\label{lem:weighofzetaissmall}
Let $\zeta \in X(T)$ and suppose $\lambda_{(0)}+p\zeta  -(p-1)\rho \in St_1$. Then $0 \leq \langle \zeta, \alpha^\vee \rangle \leq 1$ for all $\alpha \in S$.
\end{lem}
\begin{proof}
Since $\lambda_{(0)} \in X_1(T)$ we know $0 \leq \langle \lambda_{(0)}, \alpha^\vee \rangle <p.$ We also know the  weights in $St_1$ lie between $-(p-1)\rho$ and $(p-1)\rho$. Now just check the corresponding inequalities for  $\langle \zeta, \alpha^\vee \rangle$.
\end{proof}

\begin{lem}
\label{lem: E11=0}
The $E_2^{1,1}$ term:

$$E_2^{1,1}=\Ext^1_{B/B_1}\left(L(p\tau) \otimes (-p^2\mu), \Ext^1_{B_1}(L(\lambda_{(0)},k)\right)$$in \eqref{eq:LHSB1inB} is zero.
\end{lem}
\begin{proof}
Actually we will show more, namely that 
\begin{equation}
\label{eq: vanishingeachweight}
\Ext^1_{B/B_1}\left(L(p\tau) \otimes (-p^2\mu), p\zeta\right)=0
 \end{equation}
for every weight $p\zeta$ in $\Ext^1_{B_1}(L(\lambda_{(0)},k)$.
So let $p\zeta$ be a weight in $\Ext^1_{B_1}(L(\lambda_{(0)}), k)$ and consider

\begin{equation}
\label{eq:termwemustshowiszeroinE11}
\Ext^1_{B/B_1}(L(p\tau) \otimes (-p^2\mu), p\zeta) \cong \Ext^1_B(L(\tau), \zeta+ p\mu)
\end{equation}
which we must prove vanishes. Set $\tau= \tau_{(0)} + p \gamma$ with $\tau_{(0)} \in X_1(T)$ and apply the spectral sequence \eqref{eq: LHSspectralsequence}:

\begin{equation}
\label{eq:LHSinanalyszingE11}
E_2^{i,j}=\Ext^i_{B/B_1}(L(p\gamma) \otimes (-p\mu), \Ext^j_{B_1}(L(\tau_{(0)}), \zeta) \Rightarrow \Ext^{i+j}_B(L(\tau), \zeta+p\mu).
\end{equation}

To prove \eqref{eq:termwemustshowiszeroinE11} is zero, it is sufficient to show both the $E_2^{1,0}$ and $E_2^{0,1}$ terms are zero in \eqref{eq:LHSinanalyszingE11} and apply Proposition \ref{prop: Spectralsequencesummary}(b).
 Recall that $\mu \vdash d$. The assumption that $\lambda_{(0)} \neq 0$ implies that $\tau \vdash m<pd$ and so $\gamma \vdash c<d$.
 Notice from Proposition \ref{prop: HeadofLlambdaasBandBrmodule} that $E_2^{1,0}$ is immediately zero unless $\tau_{(0)}=\zeta$ in which case it is $$\Ext^1_{B/B_1}(L(p\gamma), p\mu) \cong \Ext^1_B(L(\gamma), \mu).$$
   This is also zero since $\gamma \ngeq \mu$, since $\gamma$ is a partition of a smaller integer than $\mu$ is. Thus we have shown $E_2^{1,0}=0$ in \eqref{eq:LHSinanalyszingE11}.

Finally consider $\Ext^1_{B_1}(L(\tau_{(0)}), \zeta)$ in the $E_2^{0,1}$ term:
 $$E_2^{0,1}=\Hom_{B/B_1}\left (L(p\gamma) \otimes -p\mu, \Ext^1_{B_1}(L(\tau_{(0)}),\zeta) \right).$$
 of \eqref{eq:LHSinanalyszingE11}. If it is nonzero then  $p\chi:=p\gamma - p\mu$ is a weight of $\Ext^1_{B_1}(L(\tau_{(0)}, \zeta)$ where
\begin{equation}
\label{eq: chi+mu}
\gamma=\chi + \mu.
\end{equation}
Then
\begin{equation}
\tau_{(0)} -\zeta+p\chi-(p-1)\rho \in St_1
 \end{equation}
 by Proposition \ref{prop:vanishingext1B}(c). Thus for any $\alpha \in S$:
 $$-(p-1) \leq \langle \tau_{(0)} -\zeta+p\chi-(p-1)\rho, \alpha^\vee \rangle \leq p-1$$ which yields:
 $$0 \leq \langle \tau_{(0)}-\zeta, \alpha^\vee \rangle + p\langle \chi, \alpha^\vee \rangle \leq 2(p-1).$$ But $\tau_{(0)} \in X_1(T)$ so $0 \leq \langle \tau_{(0)}, \alpha^\vee \rangle \leq p-1$. We obtain:

 $$-(p-1) \leq -\langle \zeta, \alpha^\vee \rangle + p \langle \chi, \alpha^\vee \rangle \leq 2(p-1).$$ Dividing by $p$ and rearranging gives:

 \begin{equation}
 \label{eq: crazyweights}
 -1+\frac{1}{p}+\frac{1}{p}\langle \zeta, \alpha^\vee \rangle \leq \langle \chi, \alpha^\vee \rangle \leq 2 - \frac{2}{p} + \frac{1}{p} \langle \zeta, \alpha^\vee \rangle.
 \end{equation}
 Since we are assuming $p\zeta$ is a weight in $\Ext^1_{B_1}(L(\lambda_{(0)}), k)$, then $\lambda_{(0)} + p \zeta - (p-1)\rho \in St_1$ by \cite[Prop. 3.2]{AndersenExtensionsofmodulesforalgebraicgroups}. So by Lemma \ref{lem:weighofzetaissmall}, we have $0 \leq \langle \zeta, \alpha^\vee \rangle \leq 1.$ Plugging this into \eqref{eq: crazyweights} we get:
 $$-1+\frac{1}{p} \leq \langle \chi, \alpha^\vee \rangle \leq 2-\frac{1}{p}$$ but $\chi$ is integral so $$0 \leq \langle \chi, \alpha^\vee \rangle \leq 1$$ and $\chi$ is dominant. Since $\chi$ is dominant, $\mu \vdash d$ and $\lambda \vdash c <d$, we have a contradiction to \eqref{eq: chi+mu}. Thus $\Ext^1_{B_1}(L(\tau_{(0)}),\zeta)=0$ and the $E_2^{0,1}$ term vanishes as well.
\end{proof}

Lemmas \ref{lem: E20=0}, \ref{lem: E02=0} and \ref{lem: E11=0} let us apply Proposition \ref{prop: Spectralsequencesummary}(b) and obtain:

\begin{prop}
\label{prop:ext2withQ=0}

Let $Q$ be as in \eqref{eq: LESBcoho}. Then $\Ext^2_B(Q, p^2\mu)=0$.
\end{prop}

Applying Lemma \ref{lem:B1cohovanish} and Proposition \ref{prop:ext2withQ=0} to Equation \ref{eq: LESBcoho} gives:

\begin{lem}
\label{lemma: extonetwistagrees}
Let $\mu \vdash d$. Then:
\begin{equation}
\label{eq: isoextHmuonetwist}
\Ext^1_B(H^0(p^2d), p^2\mu) \cong \Ext^1_B(H^0(pd)^{(1)}, p^2\mu).
\end{equation}
\end{lem}

We can now obtain the main result of this section:

\begin{thm}
\label{thm: extinjectionontwistedSpecht} Let $\lambda \vdash d$. Then there is a isomorphism:
$$ \HH^1(\Sigma_{pd}, S^{p\lambda}) \cong \HH^1(\Sigma_{p^2d}, S^{p^2\lambda}) .$$

\end{thm}
\begin{proof}
We have
\begin{eqnarray*}
  \HH^1(\Sigma_{p^2d}, S^{p^2\lambda})  &\cong & \Ext^1_{B}(H^0(p^2d), H^0(p^2\lambda)) \text{ by Proposition } \ref{prop: KNresultequatescohSpectwithGLd}(b)\\
   &\cong& \Ext^1_B(H^0(pd)^{(1)}, p^2\mu)  \text{ by } \eqref{eq: isoextHmuonetwist}\\
    & \cong & \Ext^1_B(H^0(pd), p\mu) \text{ by Lemma } \ref{lem: Bcohooftwsit agrees}\\
    & \cong &  \HH^1(\Sigma_{pd}, S^{p\lambda}).\\
\end{eqnarray*}
 where the last isomorphism is using Propositions \ref{prop: extagrees  in S or G category} and \ref{prop: ext agrees GLnGLd} to get from the group $GL_{p^2d}(k)$ to $GL_{pd}(k)$

\end{proof}

We immediately obtain a ``generic cohomology" result.

\begin{thm}
\label{thm: genericcohoSpechtdegree1}
Let $\lambda \vdash d$. Then for any $c \geq 1$ we have
$$\HH^1(\Sigma_{p^cd}, S^{p^c\lambda}) \cong \HH^1(\Sigma_{p^{c+1}d},S^{p^{c+1}\lambda}) .$$
\end{thm}

\begin{rem} A corresponding  result for $\HH^0(\Sigma_d, S^\lambda)$ follows from Theorem \ref{thm:JamestheoremonHom} since $\HH^0(\Sigma_{pd}, S^{p\lambda}) =0$ unless $\lambda=(d)$. For $\HH^1$ the generic cohomology can definitely be nonzero, see Example \ref{exm: papb} for instance.
\end{rem}

\section{A second stability result}
\label{section: stabilityresultaddinglargepower of p}
\subsection{}In this section we prove a stability result for Specht module cohomology involving  adding a large power of $p$ to the first part of the partition. This will generalize Corollary \ref{cor:2partmotivation}(b) for two-part partitions. The previous section used the fact that $H^0(d)^{(1)}$ sits nicely as a submodule in $H^0(pd)$. In this section we exploit the fact that $H^0(d) \otimes L(1)^{(r)}$ sits nicely inside $H^0(d+p^r)$ for large $r$.
For this section choose $r$ so that $p^r>d$ and let $n=d+p^r$. We wish to analyze the $GL_n(k)$ module $H^0(d) \otimes L(p^r)= H^0(d) \otimes L(1)^{(r)}$.

\begin{lem}
\label{lem: H0(d)timesLpahassimplesocle}
The module $H^0(d) \otimes L(1)^{(r)}$ has simple socle isomorphic to $L(d+p^r)$.
\end{lem}
\begin{proof}
Its constituents are all of the form $L(\lambda) \otimes L(1)^{(r)}$ where $[H^0(d) : L(\lambda)]=1$. From \eqref{eq: LHSspectralsequence} and Proposition \ref{prop: Spectralsequencesummary}(c), we have:

\begin{equation}
\label{eq: homisoforstab1}
\Hom_G(L(\lambda) \otimes L(1)^{(r)}, H^0(d) \otimes L(1)^{(r)}) \cong \Hom_{G/G_r}\left(L(1)^{(r)}, \Hom_{G_r}(L(\lambda), H^0(d) \otimes L(1)^{(r)})\right).
\end{equation}

Since $d<p^r$, constituents of $H^0(d)$ are of the form $L(\mu), \mu \in X_r(T)$. So

$$\Hom_{G_r}(L(\lambda), H^0(d)) \cong \Hom_G(L(\lambda), H^0(d))$$ which is zero unless $\lambda = (d)$. Thus the homomorphism space in \eqref{eq: homisoforstab1} is nonzero precisely when $\lambda=(d)$ and the statement about the socle follows.

\end{proof}

\begin{lem}
\label{lem:H0timespainsideH0}
There is an injection:
$$0 \rightarrow H^0(d) \otimes L(1)^{(r)} \rightarrow H^0(d+p^r).$$
\end{lem}

\begin{proof}
Equivalently that  there is a surjection from $V(d+p^r) $ onto $V(d) \otimes L(1)^{(r)}.$ Note that \cite[Lemma I.2.13]{jantzenbook2nded} implies

$$\Hom_G(V(d+p^r), V(d) \otimes L(1)^{(r)}) \cong \Hom_{B^+}(d+p^r, V(d) \otimes L(1)^{(r)})$$ which is one-dimensional. However both modules have isomorphic simple heads and are multiplicity free, so the map must be onto.

\end{proof}

For notational simplicity, if $\lambda=(\lambda_1, \lambda_2, \ldots, \lambda_n) \vdash d$ let $\lambda+p^r$ denote $(\lambda_1+p^r, \lambda_2, \ldots, \lambda_n)$. Next we consider the cokernel of the map in Lemma \ref{lem:H0timespainsideH0}.

\begin{lem}
    \label{lem: injectionH0withuntwistedcokerneL}
There is a short exact sequence
    \begin{equation}
        \label{eq: SESH0intoH0+pr}
            0 \rightarrow H^0(d)\otimes L(1)^{(r)} \rightarrow H^0(d+p^r) \rightarrow U \rightarrow 0
    \end{equation}
where if $[U: L(\mu)] \neq 0$ then $\mu_1<p^r$. In particular $\mu \not\unrhd \lambda+p^r$.
\end{lem}

\begin{proof}
The sequence comes from Lemma \ref{lem:H0timespainsideH0}. Now suppose $[H^0(d+p^r) : L(\mu)] \neq 0$ and $\mu_1 \geq p^r$. Then the $p$-adic expansion of $\mu_1$ is $\mu_1=p^r+c_{r-1}p^{r-1} + \cdots + c_0$. Since $d<p^r$ then $\mu \vdash d+p^r<2p^r$ so  $\mu_2<p^r$. However $\mu$ is supposed to be maximal with its carry pattern, since $L(\mu)$ is a composition factor of $H^0(d+p^r)$. This can not happen unless $\mu_1 - \mu_2 \geq p^r$. Thus $\mu$ is of the form $\tilde{\mu}+p^r$. Clearly $\tilde{\mu}$ is also maximal among partitions of $d$ with its carry pattern, so $[H^0(d) : L(\tilde{\mu})] \neq 0$. Thus $L(\mu)$ occurs in $H^0(d)\otimes L(1)^{(r)} $ and, since $H^0(d+p^r)$ is multiplicity free, not in $U$.

\end{proof}
From \cite[II.4.14]{jantzenbook2nded} we immediately obtain:

\begin{lem}
\label{lem: extvanish}
Let $U$ be as in \eqref{eq: SESH0intoH0+pr}. Then $\Ext^i_G(U, H^0(\lambda+p^r))=0$ for all $i$.
\end{lem}
Applying $\Hom_G(-, H^0(\lambda+p^r))$ to \eqref{eq: SESH0intoH0+pr} and using the Lemma \ref{lem: extvanish}, we obtain:

\begin{lem}
    \label{lem: extagreeH0}

    $$\Ext^1_G(H^0(d+p^r), H^0(\lambda+p^r)) \cong \Ext^1_G(H^0(d) \otimes L(1)^{(r)}, H^0(\lambda+p^r)).$$
    \end{lem}

By Lemma \ref{lem: extagreeH0} and Proposition \ref{prop: frobenius reciprocityforH0lambda}, we know that $$\HH^1(\Sigma_{d+p^r}, S^{\lambda+p^r}) \cong \Ext^1_{B}(H^0(\lambda)\otimes L(1)^{(r)}, \lambda+p^r).$$ The module $L(1)$ is just the natural representation $V$, i.e. column vectors with the natural $G$ action. Restricted to $B$, the lower triangular matrices, this module is clearly uniserial with simple socle $(0,0, \ldots, 0,1)$ and simple head $(1,0,\ldots, 0)$. As a module for $B$, $L(1)^{(r)}$ is still uniserial and we can write:

    \begin{equation}
        \label{eq: VtwistedresolutionasBmodule}
        0 \rightarrow Q \rightarrow L(1)^{(r)} \rightarrow (p^r, 0, \ldots, 0) \rightarrow 0.
    \end{equation}

\begin{lem}
    \label{lem:Bextvan}
Let $Q$ be as in \eqref{eq: VtwistedresolutionasBmodule}. Then
$$\Ext^1_B(H^0(d) \otimes Q, \lambda+p^r)=0.$$
\end{lem}

\begin{proof}
Consider $\tau=(0,0,\ldots, 0, p^r, 0 \ldots )$ a weight in $Q$. Since $p^r>d$ and $\lambda \vdash d$, it is clear that $\lambda+p^r-\tau$ is not dominant, nor is it of the form $s_\alpha . \sigma$ for $\alpha \in S$ and $\sigma \in X_+(T)$. The result follows from \cite[Prop. 2.3]{AndersenExtensionsofmodulesforalgebraicgroups}.
\end{proof}

We can now obtain a result which will imply the stability theorem.

\begin{thm}
    \label{thm: stabilityaddpforB}
$$\Ext^1_B(H^0(d) \otimes L(1)^{(r)}, \lambda+p^r) \cong \Ext^1_B(H^0(d), \lambda).$$
\end{thm}
\begin{proof}
Take \eqref{eq: VtwistedresolutionasBmodule} and tensor it by $H^0(d)$ then apply $\Hom_B(-, \lambda+p^r)$ to obtain

\begin{equation}
    \label{eq: Stab2LES}
     \cdots \rightarrow \Hom_B(H^0(d) \otimes Q, \lambda+p^r) \rightarrow \Ext^1_B(H^0(d) \otimes (p^r,0, \ldots, 0),\lambda+p^r) \rightarrow\end{equation}
      \begin{equation*}
      \Ext^1_B(H^0(d) \otimes L(1)^{(r)}, \lambda+p^r) \rightarrow \Ext^1_B(H^0(d) \otimes Q, \lambda+p^r) \rightarrow \cdots
\end{equation*}

The left hand term in \eqref{eq: Stab2LES} is zero just by considering weight spaces. The right hand term is zero by  Lemma \ref{lem:Bextvan}, so we have the result.
\end{proof}

We can now obtain the main result of this section.

\begin{thm}
\label{thm: stabilityaddingpr}
Let $\lambda \vdash d$ and $p^r>d$. Then:

$$\HH^1(\Sigma_d, S^\lambda) = \HH^1(\Sigma_{d+p^r}, S^{\lambda+p^r}).$$

\end{thm}

\begin{proof}
We have:
\begin{eqnarray*}
  \HH^1(\Sigma_d, S^\lambda) &\cong& \Ext^1_B(H^0(d), \lambda) \text{ by Proposition \ref{prop: KNresultequatescohSpectwithGLd}} \\
   &\cong& \Ext^1_B(H^0(d) \otimes L(1)^{(r)}, \lambda+p^r) \text{ by Theorem \ref{thm: stabilityaddingpr}} \\
   &\cong& \Ext^1_B(H^0(d+p^r), \lambda+p^r) \text{ by Lemma \ref{lem: extagreeH0} and Frobenius reciprocity.}\\
   & \cong & \HH^1(\Sigma_{d+p^r}, S^{\lambda+p^r}).
\end{eqnarray*}

\end{proof}

\section{Open problems and future directions}
\subsection{}
We suspect these results are the tip of a large iceberg of ``generic cohomology" type theorems for the symmetric group. In a recent paper \cite{CohenHemmerNakanoYoungmodulespreprint} of the author with Cohen and Nakano we proved generic cohomology results for Young modules $Y^\lambda$, specifically that for each $i \geq 0$, the cohomology groups $$\HH^i(\Sigma_{p^ad}, Y^{p^a\lambda})$$ stabilize for $a$ large enough (depending on $i$).

One question is whether our stability result can be explicitly realized, i.e.

\begin{prob} Given an element $0 \rightarrow S^{p\lambda} \rightarrow M \rightarrow k \rightarrow 0$ in $\HH^1(\Sigma_{pd}, S^{p\lambda})$, can one explicitly construct an extension of $S^{p^2\lambda}$ by $k$ realizing the isomorphism in Theorem \ref{thm: genericcohoSpechtdegree1}.
\end{prob}

Another obvious question is whether there is generic cohomology in all degrees. With our evidence from degrees zero and one we conjecture:

\begin{conj}
\label{prob: genericcohoinalldegrees}
Fix $i>0$. There is constant $c(i)$ such that for any $d$ with $\lambda \vdash d$ and any $a \geq c(i)$ that
$$\HH^i(\Sigma_{p^ad}, S^{p^a\lambda}) \cong \HH^i(\Sigma_{p^{a+1}d}, S^{p^{a+1}\lambda}).$$
\end{conj}

So $c(0)=c(1)=1$. One obstacle to generalizing our proof is that Proposition \ref{prop: KNresultequatescohSpectwithGLd}(ii) only holds through degree $2p-4$. There is certainly evidence that for larger $i$ one must ``twist" more times before stability for $\HH^i$ begins. For example we have the following analogue of Lemma \ref{lem: Bcohooftwsit agrees}.

\begin{prop}
\label{prop:Ext2agreesH0twisttwice}
Let $\mu \vdash d$. Then
$$\Ext^2_B(H^0(d)^{(1)}, p\mu) \cong \Ext^2_B(H^0(d)^{(2)}, p^2\mu).$$
\end{prop}

\begin{proof}
Consider the spectral sequence \eqref{eq: spectralsequencetwistingcoho} for the case  $r=1$, $M_1=H^0(d)^{(1)}$ and $M_2=p\mu$. The $E_2^{2,0}$ term is $\Ext^2_B(H^0(d)^{(2)}, p^2\mu)$, which we must  show is equal to $\Ext^2_B(H^0(d)^{(1)}, p\mu).$ The $E_2^{1,1}$ term is zero, as $\HH^1(B_1, k)=0$ as mentioned in the proof of Lemma \ref{lem: Bcohooftwsit agrees}. The $E_2^{0,2}$ term is
$$E_2^{0,2} \cong \Hom_B(H^0(d)^{(1)}, p\mu \otimes \HH^2(B_1, k)^{(-1)}).$$ The cohomology $\HH^2(B_1, k)$ was computed in \cite[Thm 5.3]{BNPSecondcohomologygroupsforFrobeniuskernels}. From this calculation one can see that the weights in $\HH^2(B_1, k)^{(-1)}$ are all of the form $-\alpha$ for $\alpha \in S$, in particular they are not of the form $p\sigma$. But every weight in $H^0(d)^{(1)} \otimes (-p\mu)$ is of this form, so the $E_2^{0,2}$ term is zero.

Finally observe that:

$$E_2^{0,1}=\Hom_B(H^0(d)^{(1)}, p\mu \otimes H^1(B_1, k)^{(-1)}) =0$$ by Proposition \ref{prop: degree1Bcoho}(b). Thus the differential joining $E_2^{0,1}$ to $E_2^{2,0}$ is zero, so $$E_\infty^{2,0}=E_2^{2,0}=\Ext^2_B(H^0(d)^{(1)}, p\mu)$$ as desired.
\end{proof}

\begin{rem}The additional twist going from Lemma \ref{lem: Bcohooftwsit agrees} to Proposition \ref{prop:Ext2agreesH0twisttwice} is indeed necessary. Without it, in the proof above  the $E_2^{0,2}$ term may be nonzero. For example if $p=5$ and $\mu=(50,25)$ then $\Hom_B(H^0(75), \mu \otimes (-\alpha_1) ) \neq 0.$ Further, observe that Proposition \ref{prop:Ext2agreesH0twisttwice} required information about  degree two $B_1$ cohomology. This certainly suggests further calculations of $\HH^i(B_r, k)$ may be of use in studying Specht module cohomology.

\end{rem}

\subsection{}
It would be nice to have a proof of Proposition \ref{prop: Homneq0impliesExt1neq0} using our approach, or perhaps more generally for any composition factor $L(\lambda)$ in $H^0(d)$. We conjecture:

\begin{conj}
\label{conj:inH0d meansextisnonzero}
Suppose $[H^0(d) : L(\lambda)] \neq 0$ for $\lambda \neq (d)$. Then $\HH^1(\Sigma_d, S^\lambda) \neq 0$.
\end{conj}
Note that Conjecture \ref{conj:inH0d meansextisnonzero} is a strengthening of Proposition \ref{prop: Homneq0impliesExt1neq0}. More generally the problem of computing $\HH^1$ is still open.

\begin{prob}
\label{prob:computingH1specht} For which $\lambda \vdash d$ is $\HH^1(\Sigma_d, S^\lambda)$ nonzero? Is it at most one-dimensional?
\end{prob}

There is another type of ``stability" which occurs, at least for $\HH^0(\sd, S^\lambda)$. Namely the following is an easy consequence of Theorem \ref{thm:JamestheoremonHom}:

\begin{lem}
Suppose $\lambda=(\lambda_1, \lambda_2, \ldots, \lambda_s)\vdash d$ and suppose $a \equiv -1$ mod $p^{l_p(\lambda_1)}$. Then
\begin{equation}
\label{eq:Homstabilizesaddingcong-1asnewfirstrow}
\HH^0(\Sigma_d, S^\lambda) \cong \HH^0(\Sigma_{d+a}, S^{(a,\lambda_1, \lambda_2, \ldots, \lambda_s)}).
\end{equation}
\end{lem}
This leads to the following

\begin{prob}
\label{prob:addingfirstrow-1stabilityonext}
Does the isomorphism in \eqref{eq:Homstabilizesaddingcong-1asnewfirstrow} hold for $\HH^i$ for any other $i>0$?
\end{prob}

\subsection{}Finally we ask for stronger results like that of Theorem \ref{thm: stabilityaddingpr}.

\begin{prob}
\label{prob: resultsaddingprmu}
Let $\lambda \vdash d$ and $\mu \vdash c$. Can one find more results that related the cohomology $\HH^i(\Sigma_d, S^\lambda)$ and $\HH^i(\Sigma_{d+cp^r}, S^{\lambda + p^r\mu})$?
\end{prob}

\bibliography{references0308}
\bibliographystyle{plain}
\end{document}